\documentclass{article}
\usepackage[utf8]{inputenc}
\usepackage[english]{babel}
\usepackage{mathtools}
\usepackage{amsmath}
\usepackage{hyperref}
\usepackage{amssymb}
\usepackage{filecontents}
\usepackage[numbers]{natbib}
\usepackage{pstricks}
\usepackage{rotating}
\usepackage{tkz-berge}
\usepackage{tikz}
\usetikzlibrary{shapes.geometric}
\usetikzlibrary{calc,arrows,shapes,snakes,automata,backgrounds,petri}
\usepackage{subfigure}
\usepackage{amsthm}
\newtheorem{thm}{\bf Theorem}[section]
\newtheorem*{remark}{Remark}
\usepackage[a4paper, total={6in, 8in}]{geometry}

\newcommand{\footremember}[2]{%
    \footnote{#2}
    \newcounter{#1}
    \setcounter{#1}{\value{footnote}}%
}

\title{Stabilized Partitioning of Metapopulations Networks}
\author{Dinesh Kumar\footremember{alley}{Email: kdinesh@iisc.ac.in} 
\and Soumyendu Raha}
\date{\small{Department of Computational \& Data Sciences,\\ Indian Institute of Science, Bangalore-560012, India}}

\begin{document}

\maketitle

\begin{abstract}
A metapopulations network is a multi-patch habitat system, where
populations live and interact in the habitat patches, and individuals
disperse from one patch to the other via dispersal connections. The loss of dispersal connections among the habitat patches can impact the
stability of the system. In this work, we determine if there exist(s)
set(s) of dispersal connections removal of which causes partitioning(s) of the metapopulations network
into dynamically stable sub-networks. Our study finds that there exists a lower bound threshold Fiedler value which guarantees the dynamical stability of the network dynamics. Necessary and sufficient mathematical conditions for finding partitions that result in sub-networks with the desired threshold Fiedler values have been derived and illustrated with examples. Although posed and discussed in the ecological context, it may be pointed out that such partitioning problems exist across any spatially discrete but connected dynamical systems with reaction-diffusion. Non-ecological examples are power distribution grids, intra-cellular reaction pathway networks and high density nano-fluidic lab-on-chip applications.  
\end{abstract}

\section{Introduction}
Network of metapopulations of species are commonly associated in the nature with fragmented and patched habitats. 
Theoretical and experimental studies (for instance, see \cite{taylor1991studying, hanski1999metapopulation} and references therein) suggest that metapopulation structures and dispersal are important in the persistence of the species. The distribution of the populations over a range of spatially discrete  patches and their dispersal among the patches affect the populations' interactions and persistence in the ecological system.   

Any interference (by nature or enforced) with the metapopulations structure have the risk of disruptive dynamics of metapopulations. It is well known that habitat destruction is the biggest cause of extinction for many species. Habitat fragmentation either by nature or human activities such as agriculture and construction of roads, railways, and fencing etc., lead to both lower dispersal rates (as distances between patches are not easily crossed) and higher extinction rates (as smaller habitat patches increase the competitions, thus support fewer species). Study \cite{tilman1994habitat} shows that even a dominant species eventually gets extinct due to its habitat destruction. Various infrastructural constructions fragment a previously continuous habitat and reduce both the quantity and quality of the habitat \cite{reijnen2000weg}. Populations have the extinction risk and a lesser chance of recolonization in the small and isolated habitats  \cite{opdam1991metapopulation,hanski1999metapopulation}.
Destruction of dispersal links that required to connect unstable local populations certainly has a threatning effect on the populations. Thus the study of dynamically stable spatial cuts, that is, the set of dispersal connections among the habitat patches, the removal of which partition the metapopulations network into two or more subnetworks with a minimal negative effect on the stability of the metapopulations dynamics, is important. Although stability of the metapopulations systems has been studied in many works, e.g., \cite{jansen2000local, tromeur2016impact, namba1999effects} partitioning under stability constraints remains a research issue.

Spatial cuts introduce perturbation to the dynamics of the metapopulations network and the necessary and sufficient mathematical conditions for identifying cuts that produce subnetworks which remain dynamically stable under this perturbation are studied in this work. We consider the dynamic stability of the metapopulations network dynamics in terms of the eigenvalues of the linearized system.
A spatial cut is obtained by the deletion of a set of edges of the undirected graph induced by the metapopulations network, 
in which the nodes represent the habitat patches and the edges represent the dispersal links between the habitat patches while the dispersal rates are given by the edge weights. It may be recalled that optimal partitioning problems have non-polynomial time complexity and 
hence we focus on mathematical conditions for characterizing a dynamically stable spatial cut rather than optimally computing it. 

The article is organized as follows. In the next section we analyze the stability of the linearized metapopulations model and investigate the role of the Fiedler value in it. This is followed by working out the necessary and sufficient mathematical conditions for a viable partitioning of metapopulations networks, which preserve the stability of the linearized dynamics in the partitioned subnetworks. 

\section{Fiedler Value and Metapopulations Stability}
We consider a system of $m$ patches, where $n$ species interact within each individual patch and species disperse along the links connecting the patches. Let the local dynamics of the $i$-th species in an unconnected $j$-th patch be governed by
\begin{equation}
\label{nonspatial}
\dot{x}_{i,j} = f_{i,j}(x_{1,j}\dots,x_{n,j}),~i=1,\dots,n,~j=1,\dots,m,
\end{equation}
where $x_{i,j}$ is the $i$-th species in the $j$-th patch and the real valued function $f_{i,j} \in C^{1}([0,\infty)^n)$ represents the dynamics of it. 
If a species has same dispersal rates in both the directions along a link between any two patches of the network and 
if there is a dispersal loss to the species, then the system (\ref{nonspatial}) is modified to obtain
\begin{eqnarray}
\label{spatial}
\dot{x}_{i,j} & = & f_{i,j}(x_{1,j}\dots,x_{n,j})- \underset{k\sim j}{\sum}d^{i}_{jk}(x_{i,j}-x_{i,k})-l_{i}x_{i,j}, \nonumber \\
&& i=1,\dots,n,~j=1,\dots,m,
\end{eqnarray}
where $d^{i}_{jk}$ is the per capita successful dispersal rate of $i$-th species between patches $j$ and $k$. The parameter $l_{i}$ represents the overall dispersal loss of the $i$-th species while dispersing from a patch.

Now writing the $i$-th species' dynamics in the whole patch network together by combining the local dynamics and dispersal, we get
\begin{equation}
\label{indi-spatial}
\dot{x}_{i} = f_{i}(x_1,\dots,x_n)- \mathcal{L}_{i}x_{i}-l_{i}I_{m}x_{i},~~i=1,\dots,n,
\end{equation}
where $x_{i}=(x_{i,1},\dots ,x_{i,m})^T\in \mathbb{R}^m,~f_{i}=(f_{i,1},\dots,f_{i,m})^T : R^{m\times n} \to \mathbb{R}^m$ and $\mathcal{L}_{i}\in M_m(\mathbf{R})$ is the Laplacian matrix of the network, which represents the dispersal of the $i$-th species among the patches. $I_{m}$ denotes the identity matrix of dimension $m \times m$.

With $x=(x_1,\dots,x_n)^T \in \mathbb{R}^{m \times n},~f=(f_1,\dots,f_n)^T :  R^{m\times n} \to \mathbb{R}^{m \times n},~~L=\mathcal{L}_1\bigoplus\dots\bigoplus \mathcal{L}_{n} \in M_{mn}(\mathbf{R})$ and $E=l_{1}I_{m}\bigoplus\dots\bigoplus l_{n}I_{m}\in M_{mn}(\mathbf{R})$ the overall multipopulation dynamics of all species in the patched ecosystem can written as follows.
\begin{equation}
\label{complete-spatial}
\dot{x} = f(x)- Lx -Ex.
\end{equation}

Here, all the species populations in the metapopulations network collectively represented by $x$, and $f(x),~L$ and $E$ represent its internal patch dynamics, inter-patch dispersal and dispersal loss respectively.

We are interested here in the local perturbation behavior of this system at its co-existential equilibrium solution in terms of how this perturbation 
eventually dies out or grows. If the perturbation does not grow with time, then the species populations in the spatial patch system are deemed to be locally stable around the co-existential equilibrium point. Otherwise it is considered to be unstable.

Let $\Bar{x}(t)=\Bar{x} \in \mathbb{R}^{m \times n},~t \in [0,\infty)$ be a non-trivial (component wise positive) equilibrium solution of the system (\ref{complete-spatial}). Ecologically, it means that all $n$ species populations can co-exist with the size $\Bar{x}$ and the size will remain as it is as time passes.

Let $\varepsilon(t)$ be the perturbation to the equilibrium solution $\Bar{x}$ at time $t$ so that $x(t):=\Bar{x}+\varepsilon(t)$, putting which into the system (\ref{complete-spatial}), we get 

\begin{equation}
\label{perturb-complete-spatial}
\dot{\varepsilon} = f(\Bar{x}+\varepsilon)- L(\Bar{x}+\varepsilon) -E(\Bar{x}+\varepsilon)
\end{equation}

Now by using Taylor expansion around $\Bar{x}$ and ignoring the higher order terms, the above system can be written as  

\begin{equation}
\label{linearized}
\dot{\varepsilon}= (Df(\Bar{x})- L -E)\varepsilon,
\end{equation}

To simplify the system (\ref{linearized}), let $\varepsilon=Py$ (change of variables), where $P=P_1\bigoplus\dots\bigoplus P_{n} \in M_{nm}(\mathbf{R})$ is partitioned conforming to $L$ and the columns of each $P_{i}$ consists of eigenvectors of the Laplacian $\mathcal{L}_{i}$. Since the Laplacian matrix is a symmetric matrix, the construction of a $P$ is always possible. Thus the system (\ref{linearized}) can be written as  

\begin{equation*}
P\dot{y} = Df(\Bar{x})Py- LPy-EPy,
\end{equation*}
\begin{equation*}
\implies \dot{y} = P^{-1}Df(\Bar{x})Py- P^{-1}LPy-Ey,
\end{equation*}
that is,
\begin{equation}
\label{FinalSystem}
\dot{y} = (P^{-1}Df(\Bar{x})P- \Lambda-E)y,
\end{equation}

Since the matrix $L$ is diagonalizable, $\Lambda$ is the block diagonal matrix conformal with $L$ each block diagonal entries of which are the eigenvalues of the corresponding block of the matrix $L$. There are total $mn$ equations in the system (\ref{FinalSystem}) ($m$ equations for each population or $n$ equations for each patch), out of which $n$ equations corresponds to $0$ eigenvalue of $L$ and others correspond to the positive eigenvalues of $L$.

If the real parts of the eigenvalues of the coefficient matrix $P^{-1}Df(\Bar{x})P- \Lambda-E$ in above system (\ref{FinalSystem}) are either negative or zero, then the perturbation $\varepsilon$ to the equilibrium solution $\Bar{x}$ does not grow over time, and hence the non-trivial equilibrium solution $\Bar{x}$ of the dispersal system (\ref{complete-spatial}) is stable.

One of the necessary conditions for eigenvalues of the matrix $(P^{-1}Df(\Bar{x})P- \Lambda-E)$ to have non-positive real part is the condition {\it tr}$(P^{-1}Df(\Bar{x})P- \Lambda-E) \leq 0$, which can be ensured by having $$\lambda_{2}\geq \frac{1}{n(m-1)}\underset{q}{\sum} (P^{-1}Df(\Bar{x})P-E)_{qq},$$
where $\lambda_{2}=\underset{i}{\min}\{\lambda^{i}_{2}: ~\lambda^{i}_{2}~ \text{is the second smallest eigenvalue of} ~\mathcal{L}_{i}$\} and it corresponds to the species which have minimum Fiedler value (second smallest eigenvalue) of its Laplacian matrix. 

By the Gershgorin disc theorem \cite{bhatia2013matrix}, if the following conditions hold true,
\begin{enumerate}
    \item  $l_{q}-(P^{-1}Df(\Bar{x})P)_{qq}\geq \underset{r\neq q}{\sum} |(P^{-1}Df(\Bar{x})P)_{qr}|$ (where $q$-th row of the coefficient matrix corresponds to zero eigenvalue of $L$) and 
    \item  $\lambda_{2}+l_{s}-(P^{-1}Df(\Bar{x})P)_{ss}\geq \underset{t\neq s}{\sum} |(P^{-1}Df(\Bar{x})P)_{st}|$ (where $s$-th row of the coefficient matrix corresponds to a positive eigenvalue of $L$)
\end{enumerate}
then the system (\ref{complete-spatial}) is locally stable around the non-trivial equilibrium solution $\Bar{x}$. 

Since the Fiedler value $\lambda_2$ signifies the connectivity of a graph network, it is greater than zero if and only if the graph is connected (a path between any two graph nodes exists). More the Fiedler value of a graph is, the more strongly connected the graph is. The dispersal connections in the metapopulations structure are known to be important for the metapopulations stability, and as an example given later shows, their absence causes the instability in the system. Hence it truly makes sense that both the necessary and sufficient conditions for linear stability of the metapopulations require the Fiedler value to be greater than or equal to some threshold level. For the dynamic stability of the metapopulations system we thus require the metapopulations network be sufficiently connected and therfore to have viable partition(s), metapopulations network should rich in strongly enough connected components.

We denote the Fideler value threshold level by $\tau$. Assume that $\tau=\underset{s}{\max}\underset{t\neq s}{\sum} |(P^{-1}Df(\Bar{x})P)_{st}|+(P^{-1}Df(\Bar{x})P)_{ss}-l_{s}$ and the local populations dynamics is not altered after the network partitioning by any other means (since dispersal connections play no role in the internal patch dynamics), then ensuring the Fiedler value threshold level (\textit{i.e.,} $\lambda_2\geq \tau$)  is satisfied 
in each of the partitioned component networks is sufficient for dynamic stability in the component networks.

\section{Necessary and Sufficient Conditions for Stable Graph Partitioning}

We define a simple (no self-loops), undirected and connected graph, denoted by $G(V,E)$ (or simply $G$), to be stable if its Fiedler value, denoted by $\lambda_2(\mathcal{L}(G))$ (or simply as $\lambda_2(G)$), is at least some pre-set $\tau > 0$. A partition of the graph is stable if it gives rise to components (connected subgraphs) each of which is stable. In this section, we shall provide the necessary and sufficient conditions for the stable components  of a graph or stable graph partitioning. 

There are costs, namely, internal and external costs which are associated with every graph cut. If a graph $G$ is separated into two disjoint subgraphs $G_1$ and $G_2$ by some cut, then, the external and internal costs of a node $s$ of the sub-graph $G_i~(i=1,2)$ are denoted by $E_{i}(s)$ and $I_{i}(s)$ respectively and are defined as follows.

The external cost of a node $s$ is the sum of the weights of the edges (that is, dispersal rates) that connect the node $s$ with the nodes of $G\backslash G_{i}$ . That is,
\[E_{i}(s)=\underset{t\sim s}{\sum}d_{st},~ ~\text{where}~ t ~\in G\backslash G_{i}.~~~~~~~\]
The internal cost of a node $s$ is the sum of the weights of edges that connect the node $s$ with the nodes of $G_i$ itself. That is,
 \[I_{i}(s)=\underset{t \sim s}{\sum}d_{st}, ~~\text{where}~ t ~\in G_i.~~~~~~~~~\]

\begin{thm}
\label{Necessarycond}
Let $G(V,E)$ be a stable graph with size $n (\geq 4)$, and it is partitioned into $k$ stable components $G_{1}, \dots, G_{k}$ of size $n_1, \dots n_k$ respectively, such that $\sum_{i} n_{i} =n,~ n_{i}\geq 2, ~i=1,\dots, k$. Then $I^1_i +I^2_i \geq \tau,~~\forall~ i$, where $I^{j}_{i}~(j=1,2)$ is the $j$-th smallest internal cost of $G_{i}$.
\end{thm}
\begin{proof}
Let $\mathcal{L}$ is the Laplacian of the graph $G$ and $d_i,~(i=1,\dots, n)$ are the diagonal entries of $\mathcal{L}$ in non-increasing order. Let $\mathcal{L}x_{i}=\lambda_{i}x_{i},$, where $\lambda_i~(i=1,\dots, n)$ are in non-decreasing order and $x_i$s are orthonormal vectors. Note that $\mathcal{L}$ is symmetric, thus the standard unit vector $e_{j}=\overset{n}{\underset{i=1}{\sum}}a_{ij}x_{i}, (j=1,\dots,n-2)$. We  have, 
{\small
\begin{equation}
\label{equation}
\overset{n-2}{\underset{j=1}{\sum}}d_j=\overset{n-2}{\underset{j=1}{\sum}}e^{T}_{j}\mathcal{L}e_{j}=\overset{n-2}{\underset{j=1}{\sum}}\overset{n}{\underset{i=1}{\sum}}a^{2}_{ij}\lambda_{i}=\overset{n}{\underset{i=1}{\sum}}\overset{n-2}{\underset{j=1}{\sum}}a^{2}_{ij}\lambda_{i}\leq \overset{n}{\underset{i=3}{\sum}}\lambda_{i}
\end{equation}}

Note that the above inequality (\ref{equation}) is a particular case of the Schur theorem \cite{bhatia2013matrix, cvetkovic2009introduction}. This inequality arises due to following facts:  (1) $\overset{n-2}{\underset{j=1}{\sum}}a^{2}_{ij}$  is the $i$-th diagonal entry of the matrix $AA^T, ~A=(a_{ij})$, and (2) $AA^T =\begin{bmatrix}I_{n-2}&0\\0&0\end{bmatrix}$, since $I_{n-2}=E^{T}E=(XA)^T(XA)$, where $E=[e_1,\dots,e_{n-2}]$ and $X=[x_1,\dots,x_n]$. It follows that $n-2$ values of $\overset{n-2}{\underset{j=1}{\sum}}a^{2}_{ij},~(i=1,\dots,n)$ are one and other two values of it are zero.

From the equation (\ref{equation}), we have $\lambda_{1}+\lambda_{2}\leq d_{n-1}+d_{n}$ for any graph Laplacian $\mathcal{L}$. Since $\lambda_{1}(G_{i})=0,~\lambda_{2}(G_{i})\geq \tau$ and diagonal entries of $\mathcal{L}(G_{i})$ represent the internal costs of nodes of the partitioned component $G_{i}$. Thus $I^1_i +I^2_i \geq\lambda_{2}(G_{i})\geq \tau$.

\end{proof}

\begin{thm}
\label{sufficientcond}
Let $G(V,E)$ be a stable graph with size $n(\geq 4)$ and let it be partitioned into $k$ components $G_{1}, \dots, G_{k}$ of size $n_1, \dots n_k$  respectively, such that $\sum_{i} n_{i} =n,~ n_{i}\geq 2$. Let $E^{*}_{i}$ be the largest external cost of $G_{i}$ and if $E^{*}_{i}\leq \lambda_{2}(G)-\tau$, then $G_{i}$ is a stable component. 
\end{thm}
\begin{proof}
Let $\mathcal{L}$ be the Laplcian of the graph $G$ and $L_{1}, \dots, L_{k}$ be the principal submatrices of $\mathcal{L}$ corresponding to $G_{1}, \dots, G_{k}$ respectively. By the Cauchy interlacing theorem \cite{bhatia2013matrix}, we get the following relationship in terms of the second-smallest eigenvalues of $\mathcal{L}$ and its principal submatrix $L_i$ corresponding to $G_{i}$,
$$\lambda_{2}(\mathcal{L}(G))\leq \lambda_{2}(L_{i})$$
Now writing $L_{i}=\mathcal{L}(G_{i})+D_{i}$, where $\mathcal{L}(G_{i})$ is the Laplacian matrix of $G_{i}$ and $D_{i}$ is the diagonal matrix whose diagonal entries are consisting of external costs of nodes in partitioned component $G_{i}$. Applying the Weyl's inequality \cite{bhatia2013matrix}, we get
$$\lambda_{2}(\mathcal{L}(G))\leq \lambda_{2}(L_{i})=\lambda_{2}(\mathcal{L}(G_{i})+D_{i})\leq \lambda_{2}(\mathcal{L}(G_{i}))+\lambda_{\max}(D_{i})$$
Since $\lambda_{\max}(D_{i})$ is the maximum external cost $E^{*}_{i}$ and by hypothesis $E^{*}_{i}\leq \lambda_{2}(G)-\tau$, therefore from the above inequality we obtain  $\lambda_{2}(\mathcal{L}(G_{i}))\geq \tau$ which in turn implies that $G_{i}$ is a stable partitioned component.
\end{proof}

\textbf{Example:} Consider the weighted graph $G(5,6)$ shown in Figure \ref{LaTexFigure}\subref{graphsuff} with its Fiedler value $\lambda_2=3.625$. Graph cut $C_1$ separates the graph into two components $G_1=\{v_1, v_2\}$ and $G_{2}=\{v_3,v_4,v_5\}$. Both the components having a node with maximum external cost $3$. If we assume that the graph is stable, that is, $\tau \leq 3.625$, then the sufficient condition in Theorem \ref{sufficientcond} holds true as long as $\tau \leq 0.625$. Thus both $G_1$ and $G_2$ must be stable subgraphs, in fact they are, as their Fiedler values are $6$ and $4.26$ respectively. Also, in this case, one can see the necessary condition given by Theorem \ref{Necessarycond} is essentially satisfied. The sum of the two smallest internal costs of nodes in $G_1$ and $G_2$ are $6$ and $8$ respectively, and both are greater than $\tau~(\leq 3.625)$.

For $\tau>0.625$, partitioned components with the maximum external cost $3$ may or may not be stable. Corresponds to the graph cut $C_1$, both the components are stable for all values of $\tau$. Whereas corresponds to the cut $C_2$, the component $\{v_1,v_2,v_3\}$ that also has a node with maximum external cost $3$ is unstable for $\tau>2.35$, as it has the Fiedler value $2.35$.

Now considering the graph $G(4,4)$ shown in Figure \ref{LaTexFigure}\subref{Graphnecess} with the Fiedler value $\lambda_2 =0.9529$. The sum of two smallest internal costs of the component $\{v_1,v_4\}$ (produces by the cut $C$) is $0.2$, thus the component $\{v_1,v_4\}$ is unstable if $\tau >0.2$ (as the necessary condition given by Theorem \ref{Necessarycond} fails) and stable if $\tau \leq 0.2$ (as the Fiedler value of the component is $0.2$). 

\begin{figure}[!h]
\centering
\subfigure[]
{
\begin{tikzpicture}[scale=0.55,node_style/.style={circle,fill=black!50!,font=\sffamily\large\bfseries},
   edge_style/.style={draw=black},auto]
\node[node_style] (v1) at (-3,2) {$v_1$};
\node[node_style] (v2) at (4,2) {$v_2$};
\node[node_style] (v3) at (9,-1) {$v_3$};
\node[node_style] (v4) at (4,-4) {$v_4$};
\node[node_style] (v5) at (-3,-4) {$v_5$};
    \draw[edge_style]  (v1) edge node{3} (v2);
    \draw[edge_style]  (v2) edge node{2} (v3);
    \draw[edge_style]  (v3) edge node{3} (v4);
    \draw[edge_style]  (v4) edge node{5} (v5);
    \draw[edge_style]  (v5) edge node{2} (v1);
    \draw[edge_style]  (v5) edge node{1} (v2);
    \draw [line width=0.5mm, red, style=dashed ] (-4,0) .. controls (3,-0.5) and (3.5,0.5) .. (7,2)node [right] {$C_{1}$};;
    \draw [line width=0.5mm, red, style=dashed ] (-4,-2.5) .. controls (3,-2) and (3.5,-2) .. (7,-3.5)node [right] {$C_{2}$};;
    \end{tikzpicture}
    \label{graphsuff}
    }
\hspace{.5cm}
\subfigure[]
{
\begin{tikzpicture} [scale=0.55,node_style/.style={circle,fill=black!50!,font=\sffamily\large\bfseries},
   edge_style/.style={draw=black},auto]
    \node[node_style] (v1) at (-3,2) {$v_1$};
    \node[node_style] (v2) at (4,2) {$v_2$};
    \node[node_style] (v3) at (4,-4) {$v_3$};
    \node[node_style] (v4) at (-3,-4) {$v_4$};
    \draw[edge_style] (v1) edge node{1} (v2);
    \draw[edge_style] (v2) edge node{2} (v3);
    \draw[edge_style] (v3) edge node{1} (v4);
    \draw[edge_style] (v4) edge node{0.1} (v1);
    \draw [line width=0.5mm, red, style=dashed ] (-0.5,-5) -- (-0.5,3.2) node [right] {$C$};;
   \end{tikzpicture}
\label{Graphnecess}
}
\caption{Figure (a)~ $\lambda_2 (G(5,6))=3.625$ (b) $\lambda_2 (G(4,4))=0.9529$}
\label{LaTexFigure}
\end{figure}
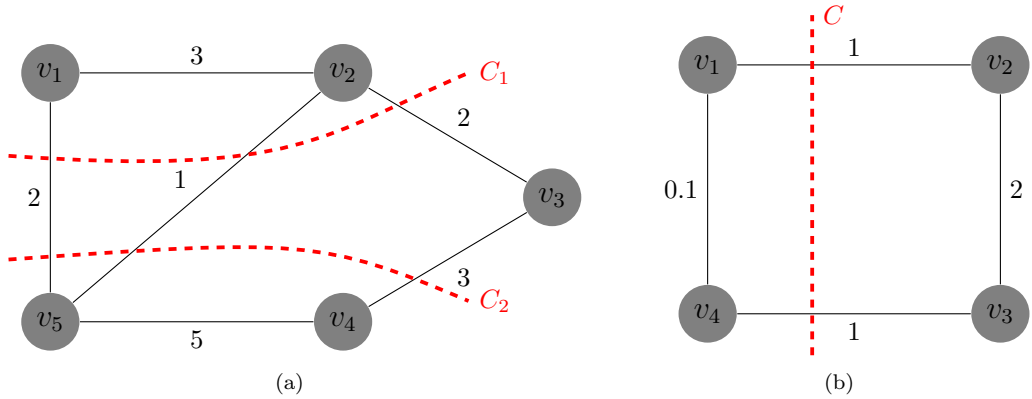

Note that the condition on maximum external cost in Theorem \ref{sufficientcond} is also sufficient to get the stable subgraph if we remove some nodes and edges associated with them. Ecologically, that corresponds to complete destruction of some habitat patches in the metapopulations structure. An alternative approach to this, is presented in the proof of next theorem while deleting some specific type of nodes.

\begin{thm}
Let $G(V,E)$ be a stable graph with size $n(\geq 3)$ and let $Y$ be a Fiedler vector of the graph Laplacian $\mathcal{L}$. Let $W$ be the non-empty set of all vertices for which the valuation of Fiedler vector is zero, that is, $W=\{v\in V : Y(v)=0\}$. If $E^{*}\leq \lambda_{2}(G)-\tau$, where $E^{*}$ is the largest external cost of a component of subgraph $G-W$ (induced by deleting the set $W$ of vertices and edges incident on them), then the component is stable.
\end{thm}
\begin{proof}
Let $L_{1},\dots,L_{k}$ be the principal submatrices of $\mathcal{L}$ corresponding to the graph components of $G-W$ and $L_{W}$ be the principal submatrix corresponding to $W$. Then the Laplacian $\mathcal{L}$ can be written as follows by using, if necessary, a permutation similarity operation. 

$$\begin{bmatrix}L_{1}&0&\cdots& 0 &\\0&L_{2}&\cdots& 0 &\\ \vdots& \vdots &\ddots & \vdots & C\\ 0&0&\cdots&L_{k} & \\ ~\\ &&C^T&&L_{W}\end{bmatrix},$$

Let $[{Y^{1}_{1}}^T~ \dots~ {Y^{k}_{1}}^T ~{Y_{2}}^T]^T$ be the conformal partition of the Fiedler vector $Y$, where $Y_{2}$ corresponds to $W$, and hence it is the zero vector. Now using the fact that $\lambda_{2}$ is the Fiedler value of $\mathcal{L}$, we have $L_{i}Y^{i}_{1}=\lambda_2Y^{i}_{1},~(i=1,\dots,k)$ which indicates $\lambda_{2}$ is also an eigenvalue of the principal submatrix $L_{i}$ corresponding to the $i$-th component of $G-W$.
We claim that either it is the smallest or the second smallest eigenvalue of $L_{i}$. Suppose contrary to it, it is the $r(>2)$-th smallest eigenvalue of $L_{i}$. Then by Cauchy interlacing theorem the $r$-th eigenvalue of $\mathcal{L}$ is less than or equal to $r$-th eigenvalue of $L_{i}$ which is $\lambda_2$, a contradiction.

Thus, $\lambda_{2}$ is less than or equal to second smallest eigenvalue of $L_{i}$, which is equal to sum of the Laplacian matrices of the $i$-th component of $G-W$ and some diagonal matrix $D$. Proceeding as in the proof of the previous theorem, using Weyl's inequality and that $E^{*}\leq \lambda_{2}(G)-\tau$, it follows that the Fiedler value of $i$-th component of $G-W$ is at least $\tau$.
\end{proof}

Our last result is about an optimal (necessary and sufficient) condition for the graph stability. It requires the evaluation of Fiedler vector, which is, in general, hard to determine without the knowledge of the Fiedler value. But for some cases (for example, regular graphs) it is easy to compute, and in those cases the following theorem is helpful in determining a stable cut.  

\begin{thm}
\label{NecessarySufficient}
Let $G(V,E)$ be a graph with size $n (\geq 2)$ and let it be partitioned into two components $G_{1}$ and $G_{2}$ of size $n_1$ and $n_2$ respectively, such that $n_{1}+n_{2} =n,~ n_{i}\geq 1,~i=1,2$. Let $Y$ be the Fiedler vector of the Laplacian $\mathcal{L}$ of $G$ corresponding to the Fiedler value $\lambda_2 (G)$. Then the graph $G$ is stable if and only if one of the following conditions
\begin{enumerate}
    \item $\underset{s\in G_{i}}{\sum}Y(s)>0~\text{and}~\underset{s\in G_{i}}{\sum}(E(s)-\tau)Y(s)\geq \underset{s\in G \backslash G_{i}}{\sum} E(s)Y(s)$
     \item $ \underset{s\in G_{i}}{\sum}Y(s)<0~\text{and}~\underset{s\in G_{i}}{\sum}(E(s)-\tau)Y(s)\leq \underset{s\in G \backslash G_{i}}{\sum} E(s)Y(s)$
      \item $ \underset{s\in G_{i}}{\sum}Y(s)=0~\text{and}~\underset{s\in G_{i}}{\sum}E(s)Y(s)= \underset{s\in G \backslash G_{i}}{\sum} E(s)Y(s)$
\end{enumerate}
is satisfied. Here $E(s)$ is the external cost of the node $s$ and $Y(s)$ is the valuation  of the Fiedler vector corresponding to the $s$-th node.
\end{thm}
\begin{proof}
Decomposing the Laplacian matrix of the graph $G$ as follows: 
\[\mathcal{L}=
\begin{bmatrix}
L_{1} & -B\\
-B^{T} & L_{2}
\end{bmatrix}
=
\begin{bmatrix}
\mathcal{L}_{1}+D_{1} & -B\\
-B^{T} & \mathcal{L}_{2}+D_{2}
\end{bmatrix},
\]
where $L_{1}$ and $L_{2}$ are the principal submatrices of the Laplacian $\mathcal{L}$ corresponding to the components $G_1$ and $G_2$ respectively, and $\mathcal{L}_{1}$ and $\mathcal{L}_{2}$ are the respective Laplacian matrices of $G_1$ and $G_2$. $D_1~(\text{or}~D_2)$ is the diagonal matrix whose $ii$-th entry is the external cost of $i$-th labelled node in the subgraph $G_1 ~(\text{or}~G_2)$. $B$ is the non-negative matrix whose $ij$-th entry is the weight of the edge $[i,j]$ if $i$-th node of $G_{1}$ connected with the $j$-th node of $G_{2}$, and $0$ otherwise. Thus column sums of the matrix $B$ are the external costs of the respective nodes of subgraph $G_{2}$.

Now considering and evaluating the term $1_{G_1}^T\mathcal{L}Y$, where $1_{G_1}$ is the column vector of size $n$ whose first $n_1$ entries are one and rest are zero, we obtain
\begin{equation}
\label{1stway}
 1_{G_1}^T(\mathcal{L}Y)=\lambda_{2}1_{G_1}^{T}Y=\lambda_{2}(G)\underset{s\in G_{1}}{\sum}Y(s)
\end{equation}
\begin{equation}
\label{2ndway}
(1_{G_1}^T\mathcal{L})Y= (1^T\mathcal{L}_1+1^{T}D_{1} -1^T B)Y_{G_1}.
\end{equation}
Here the column vector $1$ is the vector of size $n_1$ whose all entries are one and it is eigenvector corresponding to the zero eigenvalue of the matrix $\mathcal{L}_1$; $Y_{G_1}$ is conformal with $G_1$. Thus from equation (\ref{2ndway}) we have  
{\small
\begin{equation}
\label{3rd}
(1_{G_1}^T\mathcal{L})Y= 1^{T}D_{1}Y_{G_1} -1^T BY_{G_1} = \underset{s\in G_{1}}{\sum} E(s)Y(s)-\underset{s\in G_{2}}{\sum} E(s)Y(s)
\end{equation}}
Now by equations (\ref{1stway}) and (\ref{3rd}), and the fact $\lambda_2(G) \geq \tau$, we get the desired inequality for the component $G_1$.  Similarly, by choosing an appropriate column vector $1$ the inequality can be obtained for $G_2$.
\end{proof}

\begin{remark}
We can use Theorem \ref{NecessarySufficient} to check the stability of a partitioned component of a graph by considering any partition (temporarily) within the component.

\end{remark}

\section{Conclusion \& Discussion}
This paper concerns about the preservation of the linearized stability of metapopulations after partitioning the metapopulations network. Our study finds that if the internal patch (local populations) dynamics and inter-patch dispersals are appropriately conditioned, then the metapopulations dynamical system will be linearly stable around its co-existential equilibrium solution. The Fiedler value of a metapopulations network that satisfies the threshold criterion has been shown to be sufficient for the metapopulations stability provided the local populations dynamics satisfy the conditions derived 
in this work.

If the intra-patch dynamics of a species remains the same after a partition, then the cut corresponding to which the partitioned components satisfy the Fiedler value threshold criteria does not produce any growing disturbance to the existing species. Hence, such partitions effectively do not alter the ecological dynamics for the existing species in the network. Population and conservation biologists may find such a cut useful consideration in their strategy toward minimizing human-induced habitat destruction and for protecting and maintaining the species in the metapopulations structure. 

The present work provides necessary and sufficient conditions toward obtaining ecologically desirable partitions of the metapopulations networks. These conditions are easy to implement for validating or discarding a given partition and require computing only the external and internal costs. Graph-partitioning algorithms (such as \cite{kumar2019partitioning} and min-cut algorithms) combined with checking for stability conditions of the cut as obtained in this work can be used for detecting such ecologically sustainable partitions. Since optimal graph partitioning is NP-hard\cite{garey2002computers}, this would not necessarily produce a optimal partitions (unless the network size is small) but one can obtain useful solutions and approximations. 

With the considered model, the linearized stability criteria mathematically corroborate all experimental and empirical studies that have concluded that the population dispersals among patches have important role in persistence and/or stabilization of the metapopulations. For metapopulations stability, as discussed in this work, it is sufficient to have  $\lambda_2 \geq \tau$, where  the Fiedler value $\lambda_2$ signifies the connectivity of the patched network and $\tau$ is a fixed threshold level decided by the the local populations dynamics. Also, from the Weyl's monotonicity theorem \cite{bhatia2013matrix} we have the relationship between edges' weight and the Fiedler value $\lambda_2$ of the same graph network, \textit{i.e.,} any increment in the edges' weights increases the graph Laplacian eigenvalues, and in particular, increases the Fiedler value. These results in turn establish that low dispersal rates (corresponds to small $\lambda_2$) among the patches induce equilibrium instability, whereas high dispersal rates (corresponds to large $\lambda_2$) induce stability.

\appendix
\section*{Appendix}
\label{appendixEXAMPLE}
The following illustrative example is provided to show the importance of dispersal connections in the metapopulations stability of a predator-prey system.

\textbf{Example}: In a spatially homogeneous environment, consider the Rosenzweig-MacArthur predator-prey system \cite{kot2001elements} at each patch in a 3-patch spatial network (Figure \ref{Figure3patch}) and assume that at every patch, the dynamics of this system at its unique non-trivial equilibrium solution is unstable in nature. The Rosezweig-MacArthur system is defined as follows:
\begin{eqnarray}
\label{example nonspatial}
\begin{aligned}
\dot{x}_{1,j}(t) &= x_{1,j}(t)\left(1-\frac{x_{1,j}(t)}{\gamma}\right)-\frac{x_{1,j}(t)x_{2,j}(t)}{1+x_{1,j}(t)},\\ 
\dot{x}_{2,j}(t) &= \beta\left(\frac{x_{1,j}(t)}{1+x_{1,j}(t)}-\alpha\right)x_{2,j}(t),~j=1,2,3,
\end{aligned}
\end{eqnarray}
where $x_{1,j}(t)$ and $x_{2,j}(t)$ are the prey and predator density respectively at the $j$-th patch and at a time $t$. 
In the absence of predation, the prey species follows the logistic dynamics with the carrying capacity $\gamma$. The parameters $\beta$ and $\alpha$ can be manipulated to signify the conversion rate and the mortality rate of predators. The unique nontrivial equilibrium solution of (\ref{example nonspatial}) is given by 
$$(x_{1,j}^*,x_{2,j}^*)=\left(\frac{\alpha}{1-\alpha}, (1+x_{1,j}^*)(1-\frac{x_{1,j}^*}{\gamma})\right).$$

At this equilibrium, the Jacobian of the vector field of the system (\ref{example nonspatial}) is given as 

$$\begin{bmatrix}\alpha\left(1+\frac{1}{\gamma}-\frac{2}{\gamma(1-\alpha)}\right)&-\alpha\\ &\\ \beta (1-\alpha-\frac{\alpha}{\gamma}) &0\end{bmatrix}.$$

\begin{figure}
\centering
\begin{tikzpicture} [scale=0.55,node_style/.style={circle,fill=black!50!,font=\sffamily\large\bfseries},
   edge_style/.style={draw=black},auto]
    \node[node_style] (v1) at (0,-4) {$1$};
    \node[node_style] (v2) at (10,-4) {$2$};
    \node[node_style] (v3) at (5,2) {$3$};
    \draw[edge_style, below] (v1) edge node{$d^1_{12},d^2_{12}$} (v2);
    \draw[edge_style, right] (v2) edge node{$d^1_{23},d^2_{23}$} (v3);
    \draw[edge_style, left] (v3) edge node{$d^1_{13},d^2_{13}$} (v1);
\end{tikzpicture}
\caption{Figure shows the 3-patch spatial system with the dispersal connections. Dispersal rates between patches denoted by $d^1_{ij}$ and $d^2_{ij}$ are of prey and predator species respectively.}
\label{Figure3patch}
\end{figure}
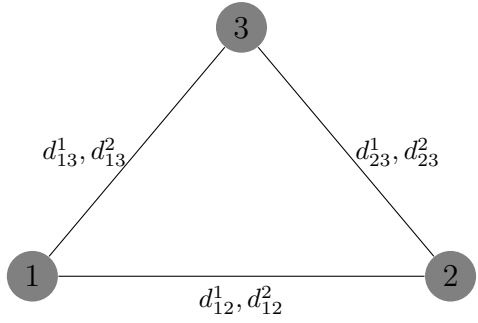

Choosing the parameters' value: $\gamma=2,~\beta=0.2,~\alpha=0.3$, the equilibrium solution $(x_{1,j}^*,x_{2,j}^*)=(3/7,55/49)$ of the system (\ref{example nonspatial}) is feasible and nontrivial, and the corresponding Jacobian has the eigenvalues $0.0107 \pm 0.1813i$. Clearly, these eigenvalues have positive real part, and thus, the dynamics at each patch is unstable around the equilibrium solution with the above choice of parameters' value.

Now considering together the Rosenzweig-MacArthur predator-prey system at each habitat patch and the species' dispersal movement among patches, the dynamics then given as
{\scriptsize
\begin{equation}
\label{example-dispersal}
\left(
\begin{array}{c}
    \dot{x}_{1,1}\\
    \dot{x}_{1,2}\\
    \dot{x}_{1,3}\\
    \dot{x}_{2,1}\\
    \dot{x}_{2,2}\\
    \dot{x}_{2,3}\\
\end{array}   \right)
= \left(\begin{array}{c}
     x_{1,1}\left(1-l_{1}-\frac{x_{1,1}}{\gamma}\right)-\frac{x_{1,1}x_{2,1}}{1+x_{1,1}}\\
    x_{1,2}\left(1-l_{1}-\frac{x_{1,2}}{\gamma}\right)-\frac{x_{1,2}x_{2,2}}{1+x_{1,2}}\\
    x_{1,3}\left(1-l_{1}-\frac{x_{1,3}}{\gamma}\right)-\frac{x_{1,3}x_{2,3}}{1+x_{1,3}}\\
     \beta\left(\frac{x_{1,1}}{1+x_{1,1}}-\alpha-l_{2}\right)x_{2,1}\\
    \beta\left(\frac{x_{1,2}}{1+x_{1,2}}-\alpha-l_{2}\right)x_{2,2}\\
    \beta\left(\frac{x_{1,3}}{1+x_{1,3}}-\alpha-l_{2}\right)x_{2,3}\\
\end{array}   \right) 
-\begin{bmatrix}
\mathcal{L}_{1}&\bf{0}\\
\bf{0}& \mathcal{L}_{2}
\end{bmatrix}
\left(
\begin{array}{c}
    x_{1,1}\\
    x_{1,2}\\
    x_{1,3}\\
    x_{2,1}\\
    x_{2,2}\\
    x_{2,3}\\
\end{array}   \right).
\end{equation}}

Here $\mathcal{L}_{1}=\begin{bmatrix}
d^1_{12}+d^1_{13}&-d^1_{1,2}&-d^1_{13}\\
-d^1_{12}&d^1_{12}+d^1_{23}&-d^1_{23}\\
-d^1_{13}&-d^1_{23}&d^1_{13}+d^1_{23} \end{bmatrix}$ and $\mathcal{L}_{2}=\begin{bmatrix}
d^2_{12}+d^2_{13}&-d^2_{1,2}&-d^2_{13}\\
-d^2_{12}&d^2_{12}+d^2_{23}&-d^2_{23}\\
-d^2_{13}&-d^2_{23}&d^2_{13}+d^2_{23}
\end{bmatrix}$ are the Laplacian matrices for prey species $x_{1}$ and predator species $x_{2}$.

The non-trivial equilibrium solution of this system is $(\bar{x}_{1,1},\bar{x}_{1,2},\bar{x}_{1,3},\bar{x}_{2,1}, \bar{x}_{2,2}, \bar{x}_{2,3}),$ where $ \bar{x}_{1,1} =\bar{x}_{1,2} =\bar{x}_{1,3}=\frac{\alpha +l_2}{1-(\alpha+l_2)}$ and $\bar{x}_{2,1}=\bar{x}_{2,2}=\bar{x}_{2,3}=\left(1+\frac{\alpha+l_2}{1-(\alpha+l_2)}\right)\left(1-l_1-\frac{\alpha+l_2}{\gamma(1-(\alpha+l_2))}\right)$.\\

At this non-trivial equilibrium point, the Jacobian of the system (\ref{example-dispersal}) is given by 
$$\begin{bmatrix}(\alpha+l_{2})\left(1-l_{1}+\frac{1}{\gamma}-\frac{2}{\gamma(1-(\alpha+l_2))}\right)I_{3}-\mathcal{L}_{1}&-(\alpha+l_{2})I_{3}\\ &\\ \beta ((1-l_{1})(1-(\alpha+l_2))-\frac{\alpha+l_2}{\gamma})I_{3}&-\mathcal{L}_{2}\end{bmatrix},$$

Taking the parameter values $d^1_{12}=1=d^1_{13}, d^1_{23}=2, d^2_{12}=2, d^2_{13}=1=d^2_{23}$ and $l_{1}=0.4=2l_{2}$, and other parameters value same as before, the eigenvalues of Jacobian of the dispersal system (\ref{example-dispersal}) are; $-0.0114,~ -0.4386,~-3.0043,$ $~-03.4496,~-5.0004 $ and $-5.4457$ (all have negative real part). Hence, the co-existential equilibrium solution of the system with among-patch dispersal is asymptotically stable.

Another way to interpret the above analysis, a stable dynamical system on the graph (in Figure \ref{Figure3patch}) becomes unstable on its components, when a graph-cut deletes all the three edges. 

\subsection*{Acknowledgements}
The first author acknowledges the financial support by University Grants Commission (UGC), India through Dr. D. S. Kothari Post Doctoral Fellowship (Ref. No. F.4-2/2006(BSR)/MA/17-18/0043).


\end{document}